\documentclass[11pt,reqno]{amsart}
\usepackage[foot]{amsaddr}

\usepackage{hyperref,url}

 \usepackage{float}
 \usepackage{mathtools}
\usepackage{amssymb}
\usepackage{bbm}
\usepackage{amssymb}
\usepackage{amsmath, amsfonts}
\usepackage{amscd}
\usepackage{amsthm}
\usepackage{setspace}
\usepackage{enumerate}
\usepackage{graphicx}
\usepackage{mathtools}
\usepackage{tcolorbox} 
\usepackage{subfigure} 
\usepackage{siunitx}
\usepackage{tikz-cd}
\usepackage{bm}
\numberwithin{equation}{section}

\usepackage{mathtools}
\usepackage[tableposition=top]{caption}
\usepackage{booktabs,dcolumn}

\DeclareMathOperator{\supp}{supp}

\newcommand{\quash}[1]{}

\def\XXint#1#2#3{{\setbox0=\hbox{$#1{#2#3}{\int}$}
     \vcenter{\hbox{$#2#3$}}\kern-.5\wd0}}

\title[Fourier inequalities and sign uncertainty]{Fourier inequalities and sign uncertainty}
\author{Roni Edwin}
\newtheorem{theorem}{Theorem}[section]

\newtheorem*{theorem1}{Theorem}
\newtheorem*{remark1}{Remark}

\newtheorem{problem}[theorem]{Problem}

\begin{document}

\begin{abstract}
    Motivated by inequalities in Fourier analysis, we present an improvement on the lower bound for the sign uncertainty principle of  Bourgain, Clozel and Kahane in high dimensions. Additionally, our methods can be used to match the existing Torquato-Stillinger lower bounds for the Cohn-Elkies linear program for sphere packing.
\end{abstract}

\maketitle

\tableofcontents

\renewcommand{\theenumi}{\alph{enumi}}
\section{Introduction}
Uncertainty is a widespread phenomenon in Fourier analysis. Loosely speaking, it says that you cannot have a large degree of control over a function and its Fourier transform. For example, the classic Heisenberg uncertainty principle says a function $f \in L^2\mathopen{}\left(\mathbb{R}\right)\mathclose{}$ and its Fourier transform cannot be both localised near points, illustrated in the inequality \begin{align}
    \left(\int_{\mathbb{R}}\left(x-x_0\right)^2\left|f(x)\right|^2\textup{d}x\right)\left(\int_{\mathbb{R}}\left(\omega-\omega_0\right)^2\bigl|\widehat{f}(\omega)\bigr|^2\textup{d}\omega\right)\ge \frac1{16\pi^2}\left(\int_{\mathbb{R}}|f|^2\right)^2.
    \label{callmebaby}
\end{align} Here we use the normalisation of the Fourier transform that makes it unitary, so \begin{align*}
    \widehat{f}(\omega)=\int_{\mathbb{R}^d}f(x)e^{-2\pi i \left\langle x,\omega\right\rangle}\textup{d}x. 
\end{align*} Notice that if $f$ and $\widehat{f}$ are localised near $x_0$ and $\omega_0$ respectively, the left-hand side of \eqref{callmebaby} would be small, contradicting the inequality. More generally, one can also consider localisation on compact sets, or sets of finite measure. For example, a result of Amrein and Berthier (Proposition $8$ in \cite{AmreinBerthiersup}) says if a function $f \in L^p\mathopen{}\left(\mathbb{R}^d\right)\mathclose{}$, $1\le p\le \infty$, and its Fourier transform $\widehat{f}$ are supported on sets of finite measure, then it must vanish identically. In the case of $p=2$, one can consequently consider uncertainty principles of the form \begin{align}
    \int_{\mathbb{R}^d}|f|^2\le K(\Sigma, S)\left(\int_{\mathbb{R}^d\setminus \Sigma}\left|f\right|^2+\int_{\mathbb{R}^d\setminus S}\bigl|\widehat{f}\bigr|^2\right),
    \label{woman}
\end{align} where $\Sigma$, $S$ are sets of finite measure. Nazarov \cite{nazarov1993est} shows that the optimal constant $K(\Sigma, S)$ above satisfies $K(\Sigma, S)\le A\exp\mathopen{}\left(A|S||\Sigma|\right)$ for some constant $A>0$, where $|X|$ denotes the Lebesgue measure of $X\subset \mathbb{R}^d$, with  Jaming \cite{jaming2007nazarov} improving this by showing the constant $K(\Sigma,S)$ is of the form \begin{align*}
    K(\Sigma,S)=C\exp\left(C\min\left(|S||\Sigma|,|S|^{\frac1{d}}w(\Sigma),|\Sigma|^{\frac1{d}}w(S)\right)\right),
\end{align*} for some constant $C>0$, where $w(X)$ denotes the mean width of $X \subset \mathbb{R}^d$.

Another type of uncertainty principle is one due to G.H. Hardy, which says that a function and its Fourier transform cannot decay too quickly, given by the following theorem: \begin{theorem1}[Hardy uncertainty principle, \cite{hardy1933theorem}]
Suppose $f\colon \mathbb{R} \to \mathbb{C}$ is a measurable function such that $|f(x)|\le Ce^{- a\pi x^2}$ and $\bigl|\widehat{f}(\omega)\bigr|\le C'e^{-\frac{\pi\omega^2}{a} }$ for all $x,\omega \in \mathbb{R}$. Then $f$ is a scalar multiple of the Gaussian $x \mapsto e^{-a \pi x^2}$.
\end{theorem1} There are also results that concern the pointwise behaviour of a function and its Fourier transform, particularly the interpolation formulas explored in \cite{radchenko2019fourier} and \cite{cohn2022universal}. For example, in \cite{radchenko2019fourier}, the authors show any even function $f \in \mathcal{S}\mathopen{}\left(\mathbb{R}\right)\mathclose{}$ is uniquely determined by the values \begin{align*}
    \left\{f\left(\sqrt{n}\right):n \in \mathbb{N}\cup\{0\}\right\}\cup \left\{\widehat{f}\left(\sqrt{n}\right):n \in \mathbb{N}\cup\{0\}\right\},
\end{align*} by providing an interpolation basis. Similar interpolation formulas are proven in \cite{cohn2022universal} for radial Schwartz functions on $\mathbb{R}^8$ and $\mathbb{R}^{24}$. The common theme among these principles is the idea that one cannot have arbitrary control over $f$ and $\widehat{f}$ in terms of decay and mass concentration.

The \emph{sign uncertainty principle,} originally investigated by Bourgain, Clozel and Kahane in \cite{bourgain2010principe} on the other hand, concerns control over the sign of $f$ and $\widehat{f}$. Their result is as follows. Observe that if $f\colon \mathbb{R} \to \mathbb{R}$ is an even function such that $f(0)\le 0$ and $\widehat{f}(0)\le 0$, then it is not possible for both $f$ and $\widehat{f}
$ to be non-negative outside an arbitrary small neighbourhood of the origin: $\widehat{f}$ is real-valued and even, since $f$ is, and the fact that \begin{align*}
    f(0)=\int_{\mathbb{R}}\widehat{f}\le 0, \ \widehat{f}(0)=\int_{\mathbb{R}} f\le 0
\end{align*} implies the quantities \begin{align*}
    A(f)&\coloneqq \inf_{r>0}\left\{f(x)\ge 0 \text{ for all }|x|\ge r\right\}, \\
    A\bigl(\widehat{f}\bigr)&\coloneqq \inf_{r>0}\left\{\widehat{f}(x)\ge 0 \text{ for all }|x|\ge r\right\},
\end{align*} are strictly positive, unless $f$ is identically $0$. Moreover the number $A(f)A\bigl(\widehat{f}\bigr)$ is invariant under dilation $x \to \lambda x$ in the sense that \begin{align*}
    A(f)A\bigl(\widehat{f}\bigr)=A\left(x \mapsto f(\lambda x)\right)A\left(\mathcal{F}\left(x \mapsto f(\lambda x)\right)\right)
\end{align*}  ($\mathcal{F}$ denotes the Fourier transform),
so it is of interest to study. The authors in \cite{bourgain2010principe} show this quantity cannot be made arbitrarily small for such applicable $f$, with the following theorem: \begin{theorem}[\cite{bourgain2010principe}, Th\' eor\`eme $1$]
    Let $f\colon \mathbb{R} \to \mathbb{R}$  be a nonzero, integrable, even function
such that $f(0),\widehat{f}(0)\le 0$, and $\widehat{f} \in L^1\mathopen{}\left(\mathbb{R}\right)\mathclose{}$. Then $A(f)A\bigl(\widehat{f}\bigr)\ge 0.1687$,
and $0.1687$ cannot be replaced by $0.41$.
\end{theorem} It is natural to consider higher dimensional generalisations of this. Formally, let $\mathcal{A}_+(d)$ denote the set of non-zero functions $f\colon \mathbb{R}^d \to \mathbb{R}$ such that $f,\widehat{f} \in L^1\mathopen{}\left(\mathbb{R}^d\right)\mathclose{}$, $\widehat{f}$ is real-valued, $f$  is eventually non-negative while $f(0)\le 0$,   and $\widehat{f}$ is eventually non-negative while $\widehat{f}(0)\le 0$. One can also consider minimising $A(f)A\bigl(\widehat{f}\bigr)$, subject to $f \in \mathcal{A}_+(d)$. This leads to the $+1$ eigenfunction uncertainty principle, as described in Problem 1.1 in \cite{cohn2019optimal}.
\begin{problem}[$+1$ uncertainty principle]
    Minimise $A(g)$ over all $g\colon \mathbb{R}^d \to \mathbb{R}$ such that \begin{enumerate}
        \item $g \in L^1\mathopen{}\left(\mathbb{R}^d\right)\mathclose{}\setminus \{\mathbf{0}\}$ and $\widehat{g}=g$,
        \item $g(0)=0$ and $g$ is eventually non-negative.
    \end{enumerate}
    \label{dancefloor}
\end{problem} Let \begin{align*}
    \mathbf{A}_+(d) \coloneqq \inf_{f \in \mathcal{A}_+(d)}\sqrt{A(f)A\bigl(\widehat{f}\bigr)}.
\end{align*} Theorem $3$ in \cite{gonccalves2017hermite} shows this infimum is equivalent to the minimal value of $A(g)$ in Problem \ref{dancefloor}, and moreover admits an extremiser.

Upper bounds for $\mathbf{A}_+(d)$ are known (\cite{gonccalves2017hermite}), but the only dimension for which the minimiser is known is $d=12$, proven by Cohn and Gon\c calves in \cite{cohn2019optimal} where $\mathbf{A}_+(12)=2$, via modular forms, which follows the approach of Viazovska in constructing `magic functions' for the sphere packing problem in dimensions $8$ \cite{viazovska2017sphere}, and $24$ \cite{cohn2017sphere}.  There is also a complementary uncertainty principle, a $-1$ uncertainty principle introduced by Cohn and Gon\c calves in \cite{cohn2019optimal}, which relates to the linear programming bounds for sphere packing. We expand on this connection later.
\begin{problem}[$-1$ uncertainty principle]
    Minimise $A(g)$ over all $g\colon \mathbb{R}^d \to \mathbb{R}$ such that \begin{enumerate}
        \item $g \in L^1\mathopen{}\left(\mathbb{R}^d\right)\mathclose{}\setminus \{\mathbf{0}\}$ and $\widehat{g}=-g$,
        \item $g(0)=0$ and $g$ is eventually non-negative.
    \end{enumerate}
    \label{dancefloor2}
\end{problem} 
We similarly call the minimal value of $A(g)$ in the above problem $\mathbf{A}_-(d)$. Like with $\mathbf{A}_+(d)$, we can alternatively define $\mathbf{A}_-(d)$ by introducing the set \begin{align*}
    \mathcal{A}_-(d)\coloneqq \left\{f \in L^1\mathopen{}\left(\mathbb{R}^d\right)\mathclose{}\setminus\{\mathbf{0}\}: \widehat{f} \in L^1\mathopen{}\left(\mathbb{R}^d\right), \ f,\widehat{f}\text{ are real-valued}\right\},
\end{align*}
and setting
\begin{align*}
    \mathbf{A}_-(d)\coloneqq \inf_{f \in \mathcal{A}_-(d)}\sqrt{A(f)A(-\widehat{f})}:
\end{align*} Given any $f \in \mathcal{A}_-(d)$, we can rescale $f$ by replacing it with $f_\lambda  \colon x\mapsto f(\lambda x)$ with $\lambda$ chosen so $A\left(f_\lambda\right)=A\bigl(\widehat{f_\lambda}\bigr)$, in which case \begin{align*}
    \sqrt{A(f)A(-\widehat{f})}=A\bigl(f_\lambda).
\end{align*} Then the function $g=f_\lambda-\widehat{f_\lambda}$ is a non-zero (otherwise the left-hand side above would not be finite) $-1$ eigenfunction, with $A(g)\le A\bigl(f_\lambda)= \sqrt{A(f)A(-\widehat{f})}$. Since the non-zero $-1$ eigenfunctions are contained in $\mathcal{A}_-(d)$, this implies $\mathbf{A}_-(d)$ as defined above is the infimum of $A(g)$ in Problem \ref{dancefloor2}.

There are some things we can note. Without loss of generality we can restrict our attention to radial functions, since the inequalities in Problems \ref{dancefloor}, \ref{dancefloor2} are invariant under rotations. One thing that makes this problem a bit finnicky is the fact that it does not lend itself well to approximations by dense subspaces, as the mapping $f \mapsto A(f)$ is not continuous in any natural sense. For example, while the extremiser for the $+1$ uncertainty principle lies in the Schwartz class $\mathcal{S}\mathopen{}\left(\mathbb{R}^d\right)\mathclose{}$ for $d=12$, it is unclear if this holds in all dimensions, or even if the infimum can be achieved by restricting our attention to a smaller class of functions, for example, Schwartz functions. Gon\c calves, Oliveira e Silva and Ramos show in \cite{gonccalves2021regularity} that that the extremiser in $d=1$ can however be achieved by sequences of functions in Schwartz space $\mathcal{S}(\mathbb{R})$.

Our interest is in the asymptotic behaviour of $\mathbf{A}_{\pm}(d)$ as $d \to \infty$. It is known from \cite{bourgain2010principe} that \begin{align*}
    \frac1{\sqrt{2\pi e}}\le \liminf_{d \to \infty} \frac{\mathbf{A}_+(d)}{\sqrt{d}}\le \limsup_{d \to \infty}\frac{\mathbf{A}_+(d)}{\sqrt{d}}\le \frac1{\sqrt{2\pi }},
\end{align*} and to date this seems to be the best asymptotic bounds. Cohn and  Gon\c calves conjecture in \cite{cohn2019optimal} that the limit $\lim_{d \to \infty}\frac{\mathbf{A}_+(d)}{\sqrt{d}}$ exists, and provide some numerical evidence to support this conjecture. Our first result is an improvement of the asymptotic lower bound for the $\pm 1$ sign uncertainty principle, motivated by  certain Fourier inequalities. The result is as follows:
\begin{theorem}
For $s \in \{\pm 1\}$, let $\mathbf{A}_s(d)$ be as defined previously. Then for all $d\ge 5$, \begin{align*}
    \frac{\mathbf{A}_s(d)}{\sqrt{d}}\ge \frac1{2\sqrt{\pi}}\approx 0.282095.
\end{align*}
    \label{bodya}
\end{theorem} This theorem is in turn proven by considering certain inequalities relating the $L^2$ norm of a function $f$ to $\lVert f\rVert_{L^1}$ and $\bigl\lVert \widehat{f}\bigr\rVert_{L^1}$. Specifically, we use the following theorem to prove Theorem \ref{bodya}:
\begin{theorem}
Let $\mathcal{H}^d$ be the subspace of $L^1\left(\mathbb{R}^d\right)$ defined by \begin{align}
\mathcal{H}^d\coloneqq \left\{f \in L^1\mathopen{}\left(\mathbb{R}^d\right)\mathclose{}: \widehat{f} \in L^1\mathopen{}\left(\mathbb{R}^d\right)\mathclose{}\right\}.
    \label{mathcalHdef}
\end{align}
Then
\begin{align*}
        \left\lVert f\right\rVert_{L^2}^2\le \left(\frac2{e}\right)^{\frac{d}{2}}\left\lVert f\right\rVert_{L^1} \big\lVert \widehat{f}\big\rVert_{L^1},
    \end{align*} for all $f \in \mathcal{H}^d$.
    \label{mainthm1}
 \end{theorem} 
 We are not the first to consider inequalities such as those in Theorem \ref{mainthm1} within the context of uncertainty principles in Fourier analysis; A. Wigderson and Y. Wigderson in \cite{wigderson} place a variety of uncertainty principles for the Fourier transform (and more general operators) in a broader context, largely relying on the fact that the Fourier transform has unit norm from $L^1$ to $L^\infty$, and is unitary. Among the tools they utilise are what they call `norm uncertainty principles', which are inequalities of the form \begin{align*}
    \left\lVert  Af\right\rVert_{L^q} \left\lVert f\right\rVert_{L^q}\le C \left\lVert f\right\rVert_{L^1} \left\lVert Af\right\rVert_{L^1},
\end{align*} where $A$ is a general operator of a certain class they call `$k$-Hadamard'. Our proof of Theorem \ref{bodya} uses a somewhat similar approach, following quite directly from Theorem \ref{mainthm1}:
\begin{proof}[Proof of Theorem \ref{bodya} under Theorem \ref{mainthm1}]
     As noted earlier, we can assume $g$ is an eigenfunction in either of Problems \ref{dancefloor}, \ref{dancefloor2}, specifically $g=s\widehat{g}$. We write $g$ as $g=g_+-g_-$, where $g_\sigma=\max\left(\sigma g,0\right)$, $\sigma \in \{\pm 1\}$. Furthermore, we may also take $g(0)=0$ in the $+1$ uncertainty principle case, because if $g(0)\le 0$, then $x \mapsto g(x)-g(0)e^{-\pi |x|^2}$ is a $+1$ eigenfunction, and \begin{align*}
    A\left(x \mapsto g(x)-g(0)e^{-\pi |x|^2}\right)\le A(g).
\end{align*}
We also normalise $g$ so $\left|\lvert g\right\rVert_{L^1}=1$, so the condition $\widehat{g}(0)=\int_{\mathbb{R}^d} g=0$ implies \begin{align*}
        \int_{\mathbb{R}^d}g_-=\int_{\mathbb{R}^d}g_+=\frac1{2}.
    \end{align*} Note $\supp g_-\subset B_{A(g)}$ ($B_r$ denotes the ball of radius $r$ in the relevant dimension). Applying Cauchy-Schwarz to the equality $\frac1{2}=\int_{B_{A(g)}}g_{-}$, we get \begin{align}
        \frac1{4}=\left(\int_{B_{A(g)}}g_-\right)^2\le \left(\int_{B_A(g)}\left|g_-\right|^2\right)\left(\int_{B_{A(g)}}1\right)\le \left|B_1\right|A(g)^d\int_{B_{A(g)}}|g|^2,
        \label{love} 
    \end{align} so \begin{align}
        \frac1{4}\le \lVert g\rVert_{L^2}^2\left|B_1\right|A(g)^d\le \left(\frac2{e}\right)^{\frac{d}{2}}\lVert g\rVert_{L^1}^2\left|B_1\right|A(g)^d=\left(\frac2{e}\right)^{\frac{d}{2}}\left|B_1\right|A(g)^d,
        \label{ayos}
    \end{align} from Theorem \ref{mainthm1}. Taking the infimum over all valid functions $g$ and rearranging, we get \begin{align}
        \mathbf{A}_s(d)^d\left(\frac2{e}\right)^{\frac{d}{2}}\ge \frac{\left|B_1\right|^{-1}}{4}=\frac{\Gamma\left(\frac{d}{2}+1\right)}{4\pi^{\frac{d}{2}}},
        \label{ricflair}
    \end{align} so \begin{align}
        \mathbf{A}_s(d)\ge \sqrt{\frac{e}{2\pi }}\left( \frac{\Gamma\left(\frac{d}{2}+1\right)}{4}\right)^{\frac1{d}}.
        \label{matriculate}
    \end{align} From Stirling's approximation, \begin{align*}
         \Gamma\left(\frac{d}{2}+1\right)\sim \sqrt{\pi d}\left(\frac{d}{2 e}\right)^{\frac{d}{2}},
    \end{align*} so this shows \begin{align*}
         \liminf_{d \to \infty}\frac{\mathbf{A}_s(d)}{\sqrt{d}}\ge \frac1{2\sqrt{\pi}}.
    \end{align*} More concretely, one can check that \begin{align*}
         \sqrt{\frac{e}{2\pi }}\left( \frac{\Gamma\left(\frac{d}{2}+1\right)}{4}\right)^{\frac1{d}}\ge \frac{\sqrt{d}}{2\sqrt{\pi }}
    \end{align*} for $d\ge 5$, proving Theorem \ref{bodya}.
\end{proof}

As stated earlier, there is a connection between Problem \ref{dancefloor2} and sphere packing in Euclidean space, via linear programming bounds, originally formulated in Theorem 3.2 in \cite{cohn2003new}, which provide an upper bound for the sphere packing density in $\mathbb{R}^d$.
\begin{theorem}[Theorem 3.2 in \cite{cohn2003new}, Theorem 3.3 in \cite{Cohn_2014}]
     Suppose $f\colon \mathbb{R}^d \to \mathbb{R}$ is continuous, positive-definitive  and integrable, satisfying
the following three conditions:\begin{enumerate}[i]
    \item \label{uno} $f(0)=\widehat{f}(0)>0$.
    \item \label{dos} $f(x)\le 0$ for all $|x|\ge r$, and 
    \item \label{tres} $\widehat{f}(t)\ge 0$ for all $t$ (positive-definiteness).
\end{enumerate}
Then the density of sphere packings in $\mathbb{R}^d$ is bounded above by $\left|B_1\right|\left(\frac{r}{2}\right)^d$.
\label{lpbounds}
\begin{remark1}
The original formulation in Theorem 3.2 in \cite{cohn2003new} imposed some decay conditions on $f$ and $\widehat{f}$, but this was later relaxed to $f$ simply being continuous, integrable and positive-definite in Theorem 3.3 in \cite{Cohn_2014}.
    The exact conclusion of Theorem 3.2 in \cite{cohn2003new} under the conditions \ref{uno},\ref{dos},\ref{tres} is that the centre density of any packing is bounded above by $\left(\frac{r}{2}\right)^d$. However, the centre density is simply the packing density divided by $\left|B_1\right|$, the volume of the unit ball, so the formulation above holds.
\end{remark1}
\end{theorem}
While the linear program above is important as a tool in the sphere packing problem, determining the optimal value of $r$ can be interesting in its own right, both as a problem in Fourier analysis,  and in determining how feasible this method is in solving sphere packing. For example, using a \emph{dual linear program} \cite{cohn2022dual} to the one above formulated above, Li \cite{li2025dual} was able to show the linear programming bounds are unable to prove the optimality of the best packings in dimensions 3, 4, 5.
Part of the motivation for Theorem \ref{mainthm1} was in developing hopefully improved lower bounds for the Cohn-Elkies linear program. For clarity, by the `Cohn-Elkies linear program', we refer to the following optimisation problem:
\begin{problem}
    Minimise $\left|B_1\right|\left(\frac{r}{2}\right)^d$ over all continuous, integrable functions $f \colon \mathbb{R}^d \to \mathbb{R}$ satisfying conditions \ref{uno}, \ref{dos} and \ref{tres}.
\end{problem}
Let $\Delta_{d}^{LP}$ denote this infimum, so \begin{align*}
    \Delta_d^{LP}=\inf\left\{\left|B_1\right|\left(\frac{r}{2}\right)^d: f\colon \mathbb{R}^d \to \mathbb{R} \text{ is continuous, integrable, and satisfies \ref{uno},\ref{dos} and \ref{tres}}\right\}.
\end{align*}
One consequence of Theorem \ref{mainthm1} is as follows: Suppose $f$ satisfies the constraints given in Theorem \ref{lpbounds}. Then $g=\widehat{f}-f$ is a non-zero $-1$ eigenfunction that is eventually non-negative, and $A(g)\le r$, since $g(x)\ge 0$ if $|x|\ge r$. Consequently, $\left|B_1\right|\left(\frac{r}{2}\right)^d\ge \left|B_1\right|\left(\frac{A(g)}{2}\right)^d\ge \left|B_1\right|\left(\frac{\mathbf{A}_{-}(d)}{2}\right)^d$. Taking the infimum over valid functions $f$ implies \begin{align}
    \Delta_d^{LP}\ge \left|B_1\right|\left(\frac{\mathbf{A}_{-}(d)}{2}\right)^d.
    \label{cominghome}
\end{align} From \eqref{ricflair}, this implies  
\begin{align}
    \Delta_d^{LP}\ge \frac1{4}\left(\frac{e}{8}\right)^{\frac{d}{2}}.
    \label{DeltaLPbounds}
\end{align}
One can note the decay of the lower bound for $\Delta_d^{LP}$ in \eqref{DeltaLPbounds} has the same dominant growth rate as the Torquato-Stillinger conjectured lower bounds for the Cohn-Elkies linear program given in \cite{torquato2006new}, Eqn (116).

One place we believe our bounds can be improved is in the sharpness of the constant in Theorem \ref{mainthm1}. One can observe by analysing the proof of Theorem \ref{bodya} given above that we really only need to consider radial functions in the context of Theorem \ref{mainthm1}. To that end, we present two related problems:
\begin{problem}[General case]
    For each dimension $d\ge 1$, determine \begin{align}
         w_d\coloneqq \sup_{f \in \mathcal{H}^d\setminus \{\mathbf{0}\}}\frac{\lVert f\rVert_{L^2}^2}{\lVert f\rVert_{L^1}\bigl\lVert \widehat{f}\bigr\rVert_{L^1}}.
         \label{ratio1}
    \end{align}
    \label{heaven1}
\end{problem}
As noted earlier, for the purposes of the sign uncertainty principles, we could restrict our attention to radial real-valued functions, which leads to the radial consideration of the problem above.
\begin{problem}[Radial case]
    Let $L^1_{\textup{rad}}\left(\mathbb{R}^d\right)$ denote the space of radial real-valued functions on $\mathbb{R}^d$. For each dimension $d\ge 1$, determine \begin{align}
        w^{\textup{rad}}_d\coloneqq \sup_{f \in \mathcal{H}^d\cap L^1_{\textup{rad}}\left(\mathbb{R}^d\right)\setminus \{\mathbf{0}\}}\frac{\lVert f\rVert_{L^2}^2}{\lVert f\rVert_{L^1}\bigl\lVert \widehat{f}\bigr\rVert_{L^1}}.
        \label{ratio2}
    \end{align}
    \label{heaven2}
\end{problem}
Some comments are in order. We initially thought that like seemingly similar inequalities in Fourier analysis, like the Hausdorff-Young inequality (\cite{beckner1975inequalities}, Theorem $1$), Gaussian functions ($x \mapsto  e^{-\pi |x|^2}$) would extremise the ratios in \eqref{ratio1},\eqref{ratio2}. This however turns out to be false. From \eqref{cominghome}, one can deduce that\begin{align}
    \Delta_{d}^{LP}\ge \frac1{4}\cdot \frac1{2^d}w_{d}^{-1},
    \label{bebe}
\end{align} where $w_d$ is as defined in \eqref{ratio1}. However, Cohn and Zhao in \cite{Cohn_2014} show the Cohn-Elkies Linear programming bound can be optimised to match the Kabatiansky-Levenshtein upper bound for sphere packing \cite{kabatiansky1978bounds}, which is asymptotically of the form $2^{-d\left(0.599...+o(1)\right)}$, which implies
\begin{align*}
    \Delta_d^{LP}\le 2^{-d\left(0.599...+o(1)\right)}.
\end{align*} Applied to \eqref{bebe}, this results in a bound of the form \begin{align*}
    w_d\ge 2^{\left(0.599....-1-o(1)\right)d},
\end{align*} or $w_d^{\frac1{d}}\ge 0.7573...$, asymptotically. As well as giving a lower bound for $w_d$, this also shows that Gaussians cannot be optimal in either of Problems \ref{heaven1}, \ref{heaven2}, as \begin{align*}
    \frac{\lVert x\mapsto e^{-\pi |x|^2}\rVert_{L^2}^2}{\bigl\lVert  x\mapsto e^{-\pi |x|^2}\bigr\rVert_{L^1}^2}=2^{-\frac{d}{2}}.
\end{align*}

As will be demonstrated in the subsequent section, our proof relies on establishing bounds on $\lVert \mathfrak{i}_s\rVert_{L^p \to L^1}$, where $\mathfrak{i}_s$ is the inclusion operator acting on radial $s$-eigenfunctions of the Fourier transform. We currently do not know how to directly bound this operator norm, and instead establish bounds for the ratio \begin{align*}
    \frac{\left\lVert f\right\rVert_{L^p}\bigl\lVert \widehat{f}\bigr\rVert_{L^p}}{\left\lVert f\right\rVert_{L^1}\bigl\lVert \widehat{f}\bigr\rVert_{L^1}},
\end{align*}  for functions $f \in \mathcal{H}^d$. We subsequently present the proof of Theorem \ref{mainthm1} in Section \ref{sec3} below.

\section{Proof of Theorem \ref{mainthm1}}\label{sec3}
We start by noting for any $f \in \mathcal{H}^d$,  by H\"{o}lder's inequality, \begin{align}
    \lVert f\rVert_{L^2}^2=\int_{\mathbb{R}^d}|f|^2=\lVert f\cdot f\rVert_{L^1}\le \lVert f\rVert_{L^p}\lVert f\rVert_{L^{p^*}},
    \label{evans}
\end{align} where $p^*$ denotes the H\"{o}lder conjugate of $p$: $\frac1{p}+\frac1{p^*}=1$.
At this point, we invoke the Hausdorff-Young inequality. The strongest form, due to Beckner, specifically Theorem $1$ in \cite{beckner1975inequalities}, is as follows. For each $p \in [1,2]$ and all $f \in L^1\left(\mathbb{R}^d\right)\cap L^2\left(\mathbb{R}^d\right)$, we have the following inequality: \begin{align*}
    \bigl\lVert \widehat{f}\bigl\rVert_{L^{p^*}}\le \left(\frac{p^{\frac1{p}}}{\left(p^*\right)^{\frac1{p^*}}}\right)^{\frac{d}{2}}\lVert f\rVert_{L^p},
\end{align*} with equality achieved by a Gaussian. Taking $p \in[1,2]$ and applying this to \eqref{evans}, we get \begin{align}
    \lVert f\rVert_{L^2}^2\le \left(\frac{p^{\frac1{p}}}{\left(p^*\right)^{\frac1{p^*}}}\right)^{\frac{d}{2}}\lVert f\rVert_{L^p}\bigl\lVert \widehat{f}\bigr\rVert_{L^p}.
    \label{ciara}
\end{align} For each $\theta \in [0,1]$, set \begin{align}
     \frac1{p_\theta}=1-\frac{\theta}{2},
      \label{pthetadef}
 \end{align} so $\frac1{p_\theta}$ is linearly interpolating (in $\theta$) between $1$ and $\frac1{2}$.  Note $p_\theta \in [1,2]$, so letting $p=p_\theta$ in \eqref{ciara}, we get \begin{align}
     \lVert f\rVert_{L^2}^2\le \left(\frac{p_\theta^{\frac1{p_\theta}}}{\left(p_\theta^*\right)^{\frac1{p_\theta^*}}}\right)^{\frac{d}{2}}\lVert f\rVert_{L^{p_\theta}}\bigl\lVert \widehat{f}\bigr\rVert_{L^{p_\theta}}.
     \label{Walton}
\end{align}
Under this definition, $L^{p_\theta}$ norms are log-convex, that is \begin{align*}
    \lVert g\rVert_{L^{p_\theta}}\le \lVert g\lVert_{L^{1}}^{1-\theta}\lVert g\rVert_{L^{2}}^{\theta},
\end{align*} for all $\theta$, and functions $g$ (for which the right-hand side is finite). From log-convexity, for $\theta \in [0,1]$, \begin{align*}
    \lVert f\rVert_{L^{p_\theta}} \bigl\lVert \widehat{f}\bigr\rVert_{L^{p_\theta}}&\le \lVert f\rVert_{L^1}^{1-\theta}\lVert f \rVert_{L^2}^{\theta}\bigl\lVert \widehat{f}\bigr\rVert_{L^1}^{1-\theta}\bigl\lVert \widehat{f}\bigr\rVert_{L^2}^{\theta} =\left(\lVert f\rVert_{L^1}\bigl\lVert \widehat{f}\bigr\rVert_{L^1}\right)^{1-\theta}\left(\lVert f\rVert_{L^2}\bigl\lVert \widehat{f}\bigr\rVert_{L^2}\right)^{\theta}.
\end{align*} Combining this with \eqref{Walton}, we get \begin{align*}
    \lVert f\rVert_{L^{p_\theta}} \bigl\lVert \widehat{f}\bigr\rVert_{L^{p_\theta}}\le \left(\lVert f\rVert_{L^1}\bigl\lVert \widehat{f}\bigr\rVert_{L^1}\right)^{1-\theta}\left(\frac{p_\theta^{\frac1{p_\theta}}}{\left(p_\theta^*\right)^{\frac1{p_\theta^*}}}\right)^{\frac{\theta d}{2}}   \left(\lVert f\rVert_{L^{p_\theta}} \bigl\lVert \widehat{f}\bigr\rVert_{L^{p_\theta}}\right)^{\theta},
\end{align*} so \begin{align*}
    \left(  \lVert f\rVert_{L^{p_\theta}} \bigl\lVert \widehat{f}\bigr\rVert_{L^{p_\theta}}\right)^{1-\theta}\le \left(\lVert f\rVert_{L^1}\bigl\lVert \widehat{f}\bigr\rVert_{L^1}\right)^{1-\theta}\left(\frac{p_\theta^{\frac1{p_\theta}}}{\left(p_\theta^*\right)^{\frac1{p_\theta^*}}}\right)^{\frac{\theta d}{2}} ,
\end{align*} which implies \begin{align}
    \lVert f\rVert_{L^{p_\theta}} \bigl\lVert \widehat{f}\bigr\rVert_{L^{p_\theta}}\le \left(\left(\frac{p_\theta^{\frac1{p_\theta}}}{\left(p_\theta^*\right)^{\frac1{p_\theta^*}}}\right)^{\frac{\theta}{2(1-\theta)}}\right)^d\lVert f\rVert_{L^1}\bigl\lVert \widehat{f}\bigr\rVert_{L^1}.
    \label{babya}
\end{align} We now let $\theta \to 1^-$. From \eqref{pthetadef}, $p_\theta=\frac2{2-\theta}$, and $\frac1{p_\theta^*}=\frac{\theta}{2}$, so \begin{align*}
    \frac{p_\theta^{\frac1{p_\theta}}}{\left(p_\theta^*\right)^{\frac1{p_\theta^*}}}=\left(\frac2{2-\theta}\right)^{\frac{2-\theta}{2}}\cdot \left(\frac{\theta}{2}\right)^{\frac{\theta}{2}},
\end{align*} so \begin{align*}
    \frac1{2(\theta-1)}\log\left( \frac{p_\theta^{\frac1{p_\theta}}}{\left(p_\theta^*\right)^{\frac1{p_\theta^*}}}\right)&=\frac1{2(\theta-1)}\left(\frac{2-\theta}{2}\log\left(\frac2{2-\theta}\right)+\frac{\theta}{2}\log\left(\frac{\theta}{2}\right)\right) \\
    &=\frac1{2(\theta-1)}\left(\frac{\theta}{2}\log\left(\frac{\theta}{2}\right)-\frac{2-\theta}{2}\log\left(\frac{2-\theta}{2}\right)\right).
\end{align*} We now evaluate the limit as $\theta \to  1^-$. Differentiating, \begin{align*}
    \partial_\theta\left(\frac{\theta}{2}\log\left(\frac{\theta}{2}\right)-\frac{2-\theta}{2}\log\left(\frac{2-\theta}{2}\right)\right)&=\frac1{2}\log\left(\frac{\theta}{2}\right)
    +\frac1{2}+\frac1{2}\log\left(\frac{2-\theta}{2}\right)-\frac{2-\theta}{2}\cdot \frac{-1}{2-\theta} \\
    &=1+\frac1{2}\log\left(\frac{\theta}{2}\right)+\frac1{2}\log\left(\frac{2-\theta}{2}\right).
\end{align*} Evaluating at $\theta=1$, we see that \begin{align*}
    \lim_{\theta \to 1^{-}}\left( \frac1{2(\theta-1)}\log\left( \frac{p_\theta^{\frac1{p_\theta}}}{\left(p_\theta^*\right)^{\frac1{p_\theta^*}}}\right)\right)=\frac{1-\log 2}{2},
\end{align*} which implies \begin{align*}
    \lim_{\theta \to 1^{-}}\left(\left(\frac{p_\theta^{\frac1{p_\theta}}}{\left(p_\theta^*\right)^{\frac1{p_\theta^*}}}\right)^{\frac{\theta}{2(1-\theta)}}\right)=\exp\left(-\frac{1-\log 2}{2}\right)=\sqrt{\frac2{e}}.
\end{align*} So, letting $\theta \to 1^{-}$ in both sides of \eqref{babya}, $p_\theta \to 2$ from \eqref{pthetadef}, so we get \begin{align*}
    \lVert f\rVert_{L^2}^2\le \left(\frac2{e}\right)^{\frac{d}{2}}\lVert f\rVert_{L^1}\bigl\lVert \widehat{f}\bigr\rVert_{L^1},
\end{align*} for all $f \in \mathcal{H}^d$, as desired. This completes the proof of Theorem \ref{mainthm1}.

 { \small 
	\bibliographystyle{plain}
	\bibliography{bib.bib} }

\end{document}